\documentclass[11pt,twoside]{article}
\usepackage{amsmath, amsthm, amscd, amsfonts, amssymb, graphicx, color}
\usepackage[bookmarksnumbered, colorlinks]{hyperref} \usepackage{float}
\usepackage{lipsum}
\usepackage{afterpage}
\usepackage[labelfont=bf]{caption}
\usepackage[nottoc,notlof,notlot]{tocbibind} 

\usepackage{lipsum}
\usepackage{fancyhdr}
\newtheorem{theorem}{Theorem}[section]
\newtheorem{lemma}[theorem]{Lemma}
\newtheorem{proposition}[theorem]{Proposition}
\newtheorem{corollary}[theorem]{Corollary}
\theoremstyle{definition}

\renewenvironment{proof}{{\bfseries \noindent Proof.}}{~~~~$\square$}
\makeatletter
\def\th@newremark{\th@remark\thm@headfont{\bfseries}}
\makeatletter

\begin{document}
\begin{center}

{\Large \bf 
Measures of Noncompactness in $\left(\bar{N}_{\Delta^{-}}^{q}\right)$ Summable Difference Sequence Spaces} 

{\bf  Tanweer Jalal $^*$\let\thefootnote\relax\footnote{$^*$Corresponding Author}}\vspace*{-2mm}\\
\vspace{2mm} {\small  National Institute of Technology, Srinagar} \vspace{2mm}

{\bf Ishfaq Ahmad Malik}\vspace*{-2mm}\\
\vspace{2mm} {\small   National Institute of Technology, Srinagar} \vspace{2mm}

\end{center}

\vspace{4mm}

{\footnotesize
\begin{quotation}
{\noindent \bf Abstract.} In the given paper we first introduce $\bar{N}_{\Delta^{-}}^{q}$ summable difference sequence spaces and prove some properties of these spaces. We then obtain the necessary and sufficient conditions for infinite matrices $A$ to map these sequence spaces on the spaces $c_0, c$ and $\ell_\infty$, the Hausdorff measure of noncompactness is then used to obtain the necessary and sufficient conditions for the compactness of the linear operators defined on these spaces.
\end{quotation}
\begin{quotation}


\end{quotation}}

\section{Introduction and Preliminaries}
We write $w$ for the set of all complex sequences $x =(x_k)_{k=0}^{\infty}$ and $\phi, c_0, c \text{ and } \ell_{\infty}$ 
for the sets of all finite, convergent sequences and sequences convergent to zero, and bounded respectively.   
The sequence $e$ is given by  $e=(1 ,1,1,\ldots)$ and $e^{(n)}$ is the sequence with 1 as only nonzero term at the $n$th place for each $n\in \mathbb{N}$, where $\mathbb{N}=\{0,1,2,\ldots\}$. Further by $cs$ and $\ell_1$ we denote the spaces of all sequences whose series  is convergent and absolutely convergent respectively. \\
The $\beta-$dual of a subset $X$ of $w$ is defined by 
$$X^{\beta}=\left\{a\in w:ax=(a_kx_k)\in cs~~\text{for all}~x=(x_k)\in X \right\}$$
If $A$ is an infinite matrix with complex entries $a_{nk}~$  $n,k\in\mathbb{N}$, we write $A_n$ for the sequence in the $n$th row of $A$, $A_n=( a_{nk})_{k=0}^{\infty}~n\in \mathbb{N}$ . The $A-$transform of any $x=(x_k)\in w$ is given by $Ax=\left(A_n(x)\right)_{k=0}^{\infty}$, where 
$$ A_n(x)=\sum_{k=0}^{\infty} a_{nk}x_k~~~~~~~~~n\in \mathbb{N}$$
the series on right must converge for each $n\in \mathbb{N}$.\\
If $X $ and $Y$ are subsets of $w$, we denote by $(X,Y)$, the class of all infinite matrices  that map $X$ into $Y$. So $A\in (X,Y)$ if and only if $A_n\in X^{\beta} ~,~n=0,1,2,\ldots$ and $Ax\in Y$ for all $x\in X$. The matrix domain of an infinite matrix $A$ in $X$ is defined by 
$$X_A=\left\{x\in w:Ax\in X\right\}$$
If $X$ and $Y$ are Banach Spaces, then by $\mathcal{B}(X,Y)$ we denote the set of all bounded (continuous) linear operators $L:X\rightarrow Y$ , which is itself a Banach space with operator norm $\|L\|=\sup_{x}\left \{\|L(x)\|_Y:\|x\|=1 \right\}$ for all $L\in \mathcal{B}(X,Y).$ The linear operator $L:X\rightarrow Y$ is said to be compact if the domain of $L$ is all of $X$ and every bounded sequence $(x_n)\in X$ , the sequence $\left(L(x_n)\right)$ has a sub-sequence which converges in $Y$. The operator $L\in \mathcal{B}(X,Y)$ is said to be of finite rank if $\dim R(L)<\infty$, where $R(L)$ denotes the range space of $L$. A finite rank operator is clearly compact.\\
The concept of difference sequence spaces was first introduced by Kizmaz \cite{kiz81} and later several authors studeid new sequence spaces defined by using difference operators like Mursaleen and Nouman \cite{mur10} and many more.\\
In the past, several authors studied matrix transformations on sequence spaces that are the matrix domains of the difference operator, or of the matrices of the classical methods of summability in spaces such as $\ell_p$, $c_0$, $c$,$\ell_\infty$ or others. For instance, some matrix domains of the difference operator were studied in (\cite{kiz81},\cite{sir92} ), of the Riesz matrices in \cite{alt02}, and so on.
In this paper, we first define a new summable difference sequence space as the matrix
domains $X_T$ of arbitrary triangles $\bar{N}_q$ and $\Delta^-$ and obtain it basis, $\beta$ dual of the new sequence spaces. We then find out the necessary and sufficient condition for the exists of matrix transformations and finally obtain the results related to the compactness of the linear operators on these new sequence spaces.
\section{$\bar{N}_{\Delta^{-}}^{q}$ Summable Difference Sequence Spaces}
Define the difference operator as 
\begin{align}\label{Dop}
\Delta^{-}x_k=x_{k-1}-x_{k}~~,k=0,1,2,\ldots~~\text{ where } x_{-1}=0
\end{align}
The $\Delta^{-}=\left(\delta_{nk}\right)_{n,k=0}^{\infty}$ is a triangular matrix written as  
$$\delta_{nk}=\left\{\begin{matrix}
-1& k= n\\
1& k=n-1 \\
0&k>n
\end{matrix}\right. $$
The inverse of this matrix is $S=(s_{nk})$ given as 
$$s_{nk}=\left\{\begin{matrix}
-1& 0\leq k\leq n\\
0&k>n
\end{matrix}\right. $$
Let $(q_k)_{k=0}^{\infty}$ be positive sequences and $(Q_n)_{n=0}^{\infty}$ be the sequence defined as $Q_n=\sum_{i=0}^{n} q_{i} $. The $(\bar{N},q)$ transform of of the sequence $(x_k)_{k=0}^{\infty}$  is defined as 
$$t_n={1\over Q_n}\sum_{i=0}^{\infty}q_i x_i$$
The matrix $\bar{N}_{q}$ for this transformation can be written as   
$$(\bar{N}_{q})_{nk}=\left\{\begin{matrix}
{q_{k}\over Q_{n}}& 0\leq k\leq n\\
0& k>n
\end{matrix}\right. $$
The inverse of this matrix is \cite{M1-b}
$$(\bar{N}_{q})_{nk}^{-1}=\left\{\begin{matrix}
(-1)^{n-k}{Q_{k}\over q_{n}}& n-1\leq k\leq n\\
0& 0\leq k\leq n-2, k>n
\end{matrix}\right. $$
We define the $\bar{N}_{\Delta^{-}}^{q}$ summable difference sequence spaces as   
\begin{align*}
&(\bar{N}_{\Delta^{-}}^{q})_0&=(c_0,\Delta^{-})_{\bar{N}_{q}}&=\left\{x\in w: 
 \bar{N}_{q}\Delta^{-}x=\left({1\over Q_n}\sum_{k=0}^{n}q_k \Delta^{-}x_k \right)_{n=0}^{\infty} \in c_0 \right\}\\
 &(\bar{N}_{\Delta^{-}}^{q})&=(c,\Delta^{-})_{\bar{N}_{q}}&=\left\{x\in w: 
 \bar{N}_{q}\Delta^{-}x =\left({1\over Q_n}\sum_{k=0}^{n}q_k \Delta^{-}x_k \right)_{n=0}^{\infty} \in c \right\} \\
 &(\bar{N}_{\Delta^{-}}^{q})_{\infty}&=(\ell_{\infty},\Delta^{-})_{\bar{N}_{q}}&=\left\{x\in w: \bar{N}_{q}\Delta^{-}x=\left({1\over Q_n}\sum_{k=0}^{n}q_k \Delta^{-}x_k \right)_{n=0}^{\infty} \in \ell_{\infty} \right\}
\end{align*}  

\parindent=0mm\vspace{0.00in}
For any sequence $x=(x_k)_{k=0}^{\infty} $, define $\tau=\tau(x)$ as the sequence with $n$th term given by 
\begin{align}\label{tau}
\tau_n=(\bar{N}_{\Delta^{-}}^{q})_n(x)={1\over Q_n}\sum_{k=0}^{n} q_k \Delta^{-}x_k~~~~~~~~~(n=0,1,2,\ldots)
\end{align}
 
\parindent=0mm\vspace{0.00in}
For any two sequences $x$ and $y$, the product $xy = (x_ky_k)_{k=0}^{\infty}$.

\subsection{Basis for the new sequence spaces}

\begin{proposition}\label{P1} 
[ \cite{wilansky2000summability}, 1.4.8, p.9]\\
Every triangle $T $ has a unique inverse $S=\left(s_{nk}\right)_{n,k=0}^{\infty}$ which is also a triangle, and $x=T(S(x))=S(T(x))$ for all $x\in w$.
\end{proposition}
\begin{proposition}\label{P2}
[\cite{Al-Em03}, Theorem 2.3] If $\left(b^{(n)}\right)_{n=0}^{\infty}$ is a basis of the linear metric space $(X,d)$, then $\left(S(b^{(n)})\right)_{n=0}^{\infty}$ is a basis of $Z=X_T$ with metric $d_T$ defined by $d_T(z,\bar{z})=d(T(z),T(\bar{z}))$ for all $z,\bar{z}\in Z$. Where $S $ is the inverse of the matrix $T$.
\end{proposition}

\parindent=0mm\vspace{0.00in}    
It is obvious that $(c_0,\Delta^{-})_{\bar{N}_{q}}=(c_0)_{\bar{N}_{q}\cdot \Delta^{-}}$, So the basis for new spaces are given by $\left(\bar{N}_{q}\cdot\Delta^{-}\right)^{-1}\left(e^{(n)}\right)=\left(\Delta^{-}\right)^{-1}\cdot\left(\bar{N}_{q}\right)^{-1} \left(e^{(n)}\right)$we have 
\begin{theorem}\label{T}
Let $\lambda_k=\left(\left(\bar{N}_{\Delta^{-}}^{q}\right)x\right)_k$ for all $k\in \mathbb{N}$. Define the sequence $s^{(k)}=\left\{s_{n}^{(k)}\right\}_{n\in \mathbb{N}}$ of the elements of $(c_0,\Delta^{-})_{\bar{N}_{q}}$ as 
 $$s_{n}^{k}=\left\{\begin{matrix}
\sum_{j=1}^{k}Q_j\left({1\over q_{j+1}}-{1\over q_j}\right)&0\leq k< n\\
-{Q_k\over q_k}& k=n \\
0& k>n \end{matrix}\right.~~~~,~~~~s_{n}^{(-1)}= \sum_{k=0}^{n}\left[
\sum_{j=1}^{k-1}Q_j\left({1\over q_{j+1}}-{1\over q_j}\right)+{Q_k\over q_k}\right]$$
for every fixed $k\in \mathbb{N}$. Then 
\begin{enumerate}
\item[i)] The sequence $\left\{s^{(k)}\right\}_{k\in \mathbb{N}}$ is a basis for the space $(c_0,\Delta^{-})_{\bar{N}_{q}}$ and any $x\in (c_0,\Delta^{-})_{\bar{N}_{q}}$ can be uniquely represented in the form 
$$x=\sum_{k}\lambda_ks^{(k)} $$.
\item[ii)] The set $\left\{e,s^{(k)}\right\}$ is a basis for the spaces $(c,\Delta^{-})_{\bar{N}_{q}}$  and any $x\in (c,\Delta^{-})_{\bar{N}_{q}}$ has a unique representation in the form 
$$x=ls_{n}^{(-1)}+\sum_{k}(\lambda_k-l)s^{(k)} $$
where for all $k\in \mathbb{N}$, $l=\lim_{k\rightarrow\infty}\left(\left(\bar{N}_{\Delta^{-}}^{q}\right)x\right)_k $. 
\end{enumerate} 
\end{theorem} 
\begin{proof}
Since $(X,\Delta^{-})_{\bar{N}_{q}}=(X)_{\bar{N}_{q}\cdot \Delta^{-}}$  for $X=c_0 ,c,\ell_\infty$. 
Now $e=(e^{(k)})_{k=}^{\infty}$ is the standard basis for $c$ and by  \\
Now $\bar{N}_{q} $ is a triangle and $\Delta^{-}$ is triangle so $\bar{N}_{q}\cdot \Delta^{-}$ is also a triangle and
$$\left(\bar{N}_{q}\cdot \Delta^{-}\right)^{-1}=\left(\Delta^{-}\right)^{-1}\cdot\left(\bar{N}_{q}\right)^{-1}
=\left\{\begin{matrix}
Q_k\left({1\over q_{k+1}}-{1\over q_k}\right)&0\leq k< n\\
-{Q_n\over q_n}& k=n \\
0& k>n
\end{matrix}\right.$$ 
Hence $\left\{s^{(k)}\right\}_{k\in \mathbb{N}}$ is a basis for the space $(c_0,\Delta^{-})_{\bar{N}_{q}}$ and the results i) and ii) are obvious to follow. 
\end{proof}

\textbf{Note:} We consider the standard basis to find the general results related to our sequence spaces. 

\begin{theorem} \label{T1}
The sequence spaces $ \left(\bar{N}_{\Delta^{-}}^{q}\right)_0$, $ \left(\bar{N}_{\Delta^{-}}^{q}\right)$ and $ \left(\bar{N}_{\Delta^{-}}^{q}\right)_\infty$ are BK-spaces with norm $\| ~\|_{\bar{N}_{\Delta^{-}}^{q}}$ given by
\begin{align*}
\| x\|_{\bar{N}_{\Delta^{-}}^{q}}=\sup_{n}\left| {1\over Q_n}\sum_{k=0}^{n}q_k \Delta^{-}x_k \right| 
\end{align*}  
If $Q_n\rightarrow \infty$ ($n\rightarrow \infty$), then $(\bar{N}_{\Delta^{-}}^{q})_0$ has AK, and every sequence $x=(x_k)_{k=0}^{\infty}\in (\bar{N}_{\Delta^{-}}^{q})$ has unique representation 
\begin{align}
x=le+\sum_{k}(\lambda_k-l)e^{(k)}    \label{2}
\end{align}
where $l\in \mathbb{C}$ is such that $x-le\in (\bar{N}_{\Delta^{-}}^{q})_0$ 
\end{theorem}
\begin{proof}
Since $\left(X,\Delta^{-}\right)_{\bar{N}_q}=X_{\bar{N}_q\cdot\Delta^{-}}$ for all $X=c_0,c, \ell_\infty$ and the spaces  $c_0,c, \ell_\infty$ are BK spaces with respect to natural norm [\cite{maddox88}, p.217-218] and the matrix $\bar{N}_q\cdot\Delta^{-}$  is a triangle so by Theorem 4.3.12, \cite{wilansky2000summability}, gives  $ \left(\bar{N}_{\Delta^{-}}^{q}\right)_0$, $ \left(\bar{N}_{\Delta^{-}}^{q}\right)$ and $ \left(\bar{N}_{\Delta^{-}}^{q}\right)_\infty$ are BK spaces\\
The space $(\bar{N}_{\Delta^{-}}^{q})_0$ has AK and the unique representation of elements of $ \left(\bar{N}_{\Delta^{-}}^{q}\right)$ are simply followed from Theorem 2 of \cite{Al-Em} and \cite{mal-Nq01}.
\end{proof}
\subsection{$\beta$ dual of the new spaces}
In order to find the $\beta$ dual we need the results of \cite{st1} which are \\
\begin{lemma}\label{L1}
$A\in (c_0:l_1)$ if and only if 
$$\sup_{K\in F}\left| \sum_{k\in K} a_{nk} \right|<\infty$$
\end{lemma}
\begin{lemma}\label{L2}
$A\in (c_0:c)$ if and only if 
$$\sup_{n}\sum\left|a_{nk} \right|<\infty,$$
$$\lim_{n\rightarrow\infty} a_{nk}-\alpha_k=0.$$
\end{lemma}
\begin{lemma}\label{L3}
 $A\in (c_0:\ell_\infty)$ if and only if 
$$\sup_{n}\sum\left|a_{nk} \right|<\infty,$$
\end{lemma}
\begin{theorem}\label{T2}
Let $(q_k)_{k=0}^{\infty}$ be positive sequences, $Q_n=\sum_{i=0}^{n} q_{i} $ and  $a=(a_k)\in w$ we define a matrix $C=(c_{nk})_{n,k=0}^{\infty}$ as 
 $$c_{nk}=\left\{\begin{matrix}
Q_k\left({1\over q_{k+1}}-{1\over q_k}\right)\sum_{j=k+1}^{n}a_j &0\leq k< n\\
-{Q_k a_k\over q_k}& k=n \\
0& k>n \end{matrix}\right.$$
and consider the sets 
\begin{align*}
c_1&=\left\{a\in w:\sup_{n}\sum_{k} |c_{nk}|<\infty\right\} &;
c_2&=\left\{a\in w:\lim_{n\rightarrow\infty} c_{nk} \text{ exists for each }~k\in \mathbb{N}\right\}\\
c_3&=\left\{a\in w:\lim_{n\rightarrow\infty}\sum_{k} |c_{nk}|=\sum_{k} \left|\lim_{n\rightarrow\infty} c_{nk}\right|\right\}&;
c_4&=\left\{a\in w:\lim_{n\rightarrow\infty} \sum_{k} c_{nk} \text{ exists }\right\}
\end{align*}
Then $\left[\left(\bar{N}_{\Delta^{-}}^{q}\right)_0\right]^{\beta}=c_1\cap c_2$ , 
$\left[\left(\bar{N}_{\Delta^{-}}^{q}\right)\right]^{\beta}=c_1\cap c_2\cap c_4$ and \\
$\left[\left(\bar{N}_{\Delta^{-}}^{q}\right)_\infty\right]^{\beta}=c_2\cap c_3$. 
\end{theorem}  
\begin{proof}
We prove the result for $\left[\left(\bar{N}_{\Delta^{-}}^{q}\right)_0\right]^{\beta}$ for the other two same procedure can be followed. Let  $x\in \left(\bar{N}_{\Delta^{-}}^{q}\right)_0$ then there exists a $y$ such that $y= \bar{N}_{\Delta^{-}}^{q} x$.\\
Hence 
\begin{align*}
\sum_{k=0}^{n}a_kx_k&=\sum_{k=0}^{n}a_k\left(\bar{N}_{\Delta^{-}}^{q}\right)^{-1}y_k\\
&=\sum_{k=0}^{n}a_k \left[\sum_{j=0}^{k-1}Q_j\left({1\over q_{j+1}}-{1\over q_{j}}\right) y_j-{Q_k\over q_k}y_k\right]\\
&=\sum_{k=0}^{n} \left[Q_{k-1}\left({1\over q_{k}}-{1\over q_{k-1}}\right)\sum_{j=k+1}^{n}a_j-{Q_ka_k\over q_k} \right]y_k\\
&=(Cy)_n
\end{align*}
So $ax=(a_nx_n)\in cs$ whenever $x\in \left(\bar{N}_{\Delta^{-}}^{q}\right)_0$ if and only if $Cy\in cs$ whenever $y\in c_0$.\\

Using Lemma \ref{L2} we get  $\left[\left(\bar{N}_{\Delta^{-}}^{q}\right)_0\right]^{\beta}=c_1\cap c_2$ 
In the same way we can show the other two results as well. 
\end{proof} 

By Theorem 7.2.9, \cite{wilansky2000summability} we know that if $X$ is a BK-space and $a\in w$ then 
$$\|a\|^*=\sup\left\{\left| \sum_{k=0}^{\infty} a_kx_k\right|:\|x\|=1\right\}$$
provided the term on the right side exists and is finite, which is the case whenever $a\in X^\beta$.
\begin{theorem}\label{T3}
For $\left[\left(\bar{N}_{\Delta^{-}}^{q}\right)_0\right]^{\beta}$ ,  $\left[\left(\bar{N}_{\Delta^{-}}^{q}\right)\right]^{\beta}$ and $\left[\left(\bar{N}_{\Delta^{-}}^{q}\right)_\infty\right]^{\beta}$ the norm $\|~\|^*$ is defined as

$$\|a\|^*=\sup_{n}\left[\sum_{k=0}^{n-1}Q_{k}\left|\left({1\over q_{k+1}}-{1\over q_k}\right) \sum_{j=k+1}^{n}a_j\right|+
\left|{Q_na_n\over q_n}\right|\right]$$ 
\end{theorem} 
\begin{proof}
If $x^{[n]}$ denotes the $n$th section of the sequence $x\in \left(\bar{N}_{\Delta^{-}}^{q}\right)_0$ then using (\ref{tau}) we have 
\begin{align*}
\tau_{k}^{[n]}=\tau_{k}(x^{[n]})={1\over Q_k}\sum_{j=0}^{k}q_j\Delta^{-}x_{j}^{[n]}
\end{align*}
Let $a\in \left[\left(\bar{N}_{\Delta^{-}}^{q}\right)_0\right]^{\beta}$, then for any non-negative integer $n$ define the sequence $d^{[n]}$ as 
$$d_{k}^{[n]}= \left\{\begin{matrix}
Q_k\left({1\over q_{k+1}}-{1\over q_k}\right)\sum_{j=k+1}^{n}a_j &0\leq k< n\\
-{Q_k a_k\over q_k}& k=n \\
0& k>n \end{matrix}\right.$$
Let $\|a\|_{\Pi} = \sup_{n}\| d^{[n]}\|_1=\sup_{n}\left(\sum_{k=0}^{\infty}|d_k^{[n]}|\right) $  where   $\Pi=\left[\left(\bar{N}_{\Delta^{-}}^{q}\right)\right]^{\beta}$. 
The inequality $\|a\|_{\Pi}\leq \|a\|^{*}~~~~\label{n}$ is obvious. \\
Also 
\begin{align*}
\left|\sum_{k=0}^{\infty}a_k x_k^{[n]}\right|&=\left|\sum_{k=0}^{n}a_k \left( \sum_{j=0}^{k}{1\over q_j}(Q_j\tau_{j}^{[n]}-
Q_{j-1}\tau_{j-1}^{[n]})\right)\right| \\
&\leq \left|\sum_{k=0}^{n-1} Q_k\left( {1\over q_{k+1}}-{1\over q_k}\right)\left(\sum_{j=k+1}^{n}a_j \right)\tau_{k}^{[n]}\right|+\left|{a_nQ_n\over q_n} \right||\tau_{n}^{[n]}|\\
&\leq \sup_{k}|\tau_{k}^{[n]}|\cdot \left(Q_k\left({1\over q_{k+1}}-{1\over q_k}\right)\sum_{j=k+1}^{n}a_j+ \left|a_nQ_n\over q_n \right|\right) \\
&=\|x^{[n]}\|_{\bar{N}_{\Delta^-}^{q}}\|d^{[n]}\|_1 \\
&=\|a\|_{\Pi}\|x^{[n]}\|_{\bar{N}_{\Delta^-}^{q}}
\end{align*}
Hence $\|a\|^{*}\leq \|a\|_{\Pi}   $ \\
From the above inequalities we get the required conclusion. 
\end{proof}

Following are some well known results \\
\begin{proposition}\label{P3}
(cf. \cite{mal98}, Theorem 7) Let $X$ and $Y$ be BK spaces, then $(X,Y)\subset \mathcal{B}(X,Y)$ that is every matrix $A$ from $X$ into $Y$ defines an element $L_A$ of $\mathcal{B}(X,Y)$ where 
$$L_A(x)=A(x)~~~~~~~~~~~~~~\forall~x\in X$$
Also $A\in (X,\ell_\infty)$ if and only if
$$\|A\|^*=\sup_{n}\|A_n\|^*=\|L_A
\|<\infty $$ 
If $\left(b^{(k)}\right)_{k=0}^{\infty}$ is a basis of $X , Y$ and 
$Y_1$ are FK spaces with $Y_1$ a closed subspace of $Y$, then 
$A\in (X,Y_1)$ if and only if $A\in (X,Y)$ and $A\left(b^{(k)}\right)
\in Y_1$ for all $k=0,1,2,\ldots$.
\end{proposition}

\begin{proposition}\label{P4} (cf. \cite{mal99}, Proposition 3.4) Let $T$ be a triangle\\
\begin{enumerate}
\item[(i)] If $X \text{  and  } Y$ are subsets of $w$, then $A\in (X,Y_T)$ if and only if $B=TA\in (X,Y)$.
\item[(ii)] If $X \text{  and  } Y$ are BK spaces and $A\in (X,Y_T)$, then 
\begin{align*}\label{operatorequality}
\|L_A\|=\|L_B\|
\end{align*} 
\end{enumerate} 
\end{proposition}
Using Proposition \ref{P3} and Theorem \ref{T3} we can easily conclude that following:
\begin{corollary}\label{C1}
Let $(q_k)_{k=0}^{\infty}$ be a positive sequence, $Q_n=\sum_{k=0}^{n} q_k$ and $\Delta^-$ be the difference operator as defined in \eqref{Dop}, then 
\begin{enumerate}
\item[i)] $A\in \left(\left(N_{\Delta^-}^{q} \right)_\infty, \ell_\infty\right)$  if and only if 
\begin{equation}\label{Co1i}
\sup_{m,n}\left[\sum_{k=0}^{m-1}Q_{k}\left|\left({1\over q_{k+1}}-{1\over q_k}\right) \sum_{j=k+1}^{m}a_{nj}\right|+
\left|{Q_ma_{nm}\over q_m}\right|\right]<\infty
\end{equation}
 and 
\begin{equation}\label{Co1ii}
{A_nQ\over q} \in c_0~~~\forall~n=0,1,\ldots
\end{equation}
\item[ii)] $A\in \left(\left(N_{\Delta^-}^{q} \right), \ell_\infty\right)$  if and only if condition \eqref{Co1i} holds and 
\begin{equation}\label{Co2}
{A_nQ\over q} \in c~~~~~~~\forall~n=0,1,2,\ldots
\end{equation}
\item[iii)] $A\in \left(\left(N_{\Delta^-}^{q} \right)_0, \ell_\infty\right)$  if and only if condition \eqref{Co1i} holds.
\item[iv)] $A\in \left(\left(N_{\Delta^-}^{q} \right)_0, c_0\right)$  if and only if condition \eqref{Co1i} holds and
\begin{equation}\label{Co3}
\lim_{n\rightarrow\infty} a_{nk}=0~~~~~\text{ for all  } k=0,1,2\ldots
\end{equation}
\item[v)] $A\in \left(\left(N_{\Delta^-}^{q} \right)_0, c\right)$  if and only if condition \eqref{Co1i} holds and
\begin{equation}\label{Co4}
\lim_{n\rightarrow\infty} a_{nk}=\alpha_k~~~~~\text{ for all  } k=0,1,2\ldots
\end{equation}
\item[vi)] $A\in \left(\left(N_{\Delta^-}^{q} \right), c_0\right)$  if and only if conditions \eqref{Co1i}, \eqref{Co1ii} and \eqref{Co3} holds and
\begin{equation}\label{Co5}
\lim_{n\rightarrow\infty}\sum_{k=0}^{\infty} a_{nk}=0~~~~~\text{ for all  } k=0,1,2\ldots
\end{equation}
\item[vii)] $A\in \left(\left(N_{\Delta^-}^{q} \right), c\right)$  if and only if conditions \eqref{Co1i}, \eqref{Co1ii} and \eqref{Co4} holds and
\begin{equation}\label{Co6}
\lim_{n\rightarrow\infty}\sum_{k=0}^{\infty} a_{nk}=\alpha~~~~~\text{ for all  } k=0,1,2\ldots
\end{equation}
\end{enumerate}
\end{corollary}
Again Using Proposition \ref{P4} and Theorem \ref{T1} we have the following corollary:
 
\section{Hausdorff Measure of Noncompactness}
From Mursaleen et. al. \cite{mur11} we let $S$ and $M$ be the subsets of a metric space $(X,d)$ and $\epsilon>0$. Then $S$ is called an $\epsilon-$net of $M$ in $X$ if for every $x\in M$ there exists $s\in S$ such that $d(x,s)<\epsilon$. Further, if the set $S$ is finite, then the $\epsilon-$net $S$ of $M$ is called {\it finite $\epsilon-$net} of $M$. A subset of a metric space is said to be {\it totally bounded} if it has a finite $\epsilon-$net for every $\epsilon>0.$\\
If $\mathcal{M}_X$ denotes the collection of all bounded subsets of metric space $(X,d)$. If $Q\in \mathcal{M}_X$ then the \emph{Hausdorff Measure of Noncompactness} of the set $Q$ is defined by 
$$\chi(Q)=\inf \left\{\epsilon>0: Q \text{  has a finite } \epsilon-\text{net in  }X \right\} $$
 The function $\chi:\mathcal{M}_X\rightarrow [0,\infty)$ is called \emph{Hausdorff Measure of Noncompactness} \cite{Josaf}\\
The basic properties of \emph{Hausdorff Measure of Noncompactness} can be found in (\cite{M1-b}, \cite{malkowsky1}, \cite{Josaf}). Some of those properties are \\
If $Q,Q_1$ and $Q_2$ are bounded subsets of a metric space $(X,d),$ then 
\begin{align*}
\chi(Q)&=0 \Leftrightarrow Q\text{ is totally bpunded set,}\\
\chi(Q)&=\chi(\bar{Q}),\\
Q_1\subset Q_2 &\Rightarrow \chi(Q_1)\leq \chi(Q_2),\\
\chi(Q_1\cup Q_2)&=\max\left\{ \chi(Q_1),\chi(Q_2)\right\},\\
\chi(Q_1\cap Q_2)&=\min\left\{ \chi(Q_1),\chi(Q_2)\right\}.
\end{align*}
Further if $X $ is a normed space the $\chi$ has the additional properties connected with the linear structure. \\
\begin{align*}
\chi(Q_1+ Q_2)&\leq \chi(Q_1)+\chi(Q_2)\\
\chi(\eta Q)&=|\eta|\chi(Q)~~~~~~~~~~~~~~~~\eta \in \mathbb{C}
\end{align*}
If $X$ and $Y$ are normed space, then for $A\in \mathcal{B}(X,Y)$ the Hausdorff Measure of Noncompactness of $A$, is denoted by $\|A\|_\chi$ and is defined as  
$$\|A\|_\chi=\chi(AB)$$
Where $B=\{x\in X:\|x\|=1\}$ is the unit ball in $X$. \\
Also $A $ is said to be compact if and only if  $\|A\|_\chi=0$ and $\|A\|_\chi\leq \|A\|$.
\begin{proposition} \label{P11} (\cite{Josaf}, Theorem 6.1.1, $X=c_0$)
Let $Q\in M_{c_0}$ and $P_r:c_0\rightarrow c_0 ~~(r\in \mathbb{N}$ be the operator defined by  $P_r(x)=(x_0,x_1,\ldots, x_r, 0,0,\ldots)$ for all $x=(x_k)\in c_0$. Then, we have 
$$\chi(Q)=\lim_{r\rightarrow \infty}\left(\sup_{x\in Q}\|(I-P_r)(x)\| \right)$$
where $I$ is the identity operator on $c_0$. 
\end{proposition}

\begin{proposition} \label{P12} (\cite{Josaf}, Theorem 6.1.1)
Let $X$ be a Banach space with a Schauder basis $\{e_1,e_2,\ldots \}$, and $Q\in M_{X}$ and $P_n:X\rightarrow X ~~(n\in \mathbb{N}$ be the projector onto the linear span of  $\{e_1,e_2,\ldots , e_n\}$. Then, we have 
$${1\over a}\lim_{n\rightarrow \infty}\sup \left(\sup_{x\in Q}\|(I-P_n)(x)\| \right) \leq \chi(Q)\leq \inf_{n}\left(\sup_{x\in Q}\|(I-P_n)(x)\| \right)\leq 
\lim_{n\rightarrow \infty}\sup \left(\sup_{x\in Q}\|(I-P_n)(x)\| \right)$$ 
where  $a=\lim_{n\rightarrow \infty}\sup \|I-P_n\|.$
If $X=c$ then $a=2$. (see \cite{Josaf}, p.22).
\end{proposition}
   
\section{Compact operators on the spaces $\left(N_{\Delta^-}^{q} \right)_0$, $\left(N_{\Delta^-}^{q} \right)$ and $\left(N_{\Delta^-}^{q} \right)_\infty$}
\begin{theorem}\label{T4}
Consider the matrix $A$ as in Corollary \ref{C1}, and for any integers n,s, $n>s$ set 
\begin{equation}\label{E1}
\|A\|^{(s)}=\sup_{n>p}\sup_{m}\left(\sum_{j=0}^{m-1}Q_j\left|\left({1\over q_{j+1}}-{1\over q_j}\right)\sum_{i=j+1}^{m}a_{ni} \right|+\left|{Q_ma_{nm}\over q_m} \right|\right)
\end{equation} 
If $X$ be either $\left(N_{\Delta^-}^{q} \right)_0$ or $\left(N_{\Delta^-}^{q} \right)$ and $A\in (X,c_0) $. Then 
\begin{equation}\label{E2}
\|L_A\|_{\chi}=\lim_{s \rightarrow\infty} \|A\|^{(s)}.
\end{equation} 
If $X$ be either $\left(N_{\Delta^-}^{q} \right)_0$ or $\left(N_{\Delta^-}^{q} \right)$ and $A\in (X,c) $. Then 
\begin{equation}\label{E3}
{1\over 2}\cdot \lim_{s\rightarrow\infty} \|A\|^{(s)}\leq \|L_A\|_{\chi}\leq \lim_{r\rightarrow\infty} \|A\|^{(s)}.
\end{equation}
and if $X$ be either $\left(N_{\Delta^-}^{q} \right)_0$ , $\left(N_{\Delta^-}^{q} \right)$ or $\left(N_{\Delta^-}^{q} \right)_\infty$ and $A\in (X,\ell_\infty) $. Then 
\begin{equation}
0\leq \|L_A\|_{\chi}\leq \lim_{s\rightarrow\infty} \|A\|^{(s)}.
\end{equation}
\end{theorem} 
\begin{proof}
Let $F=\{x\in X:\|x\|\leq 1\}$ if $A\in (X,c_0) $ and $X$ is one of the spaces $\left(N_{\Delta^-}^{q} \right)_0$ or $\left(N_{\Delta^-}^{q} \right)$, then by Proposition \ref{P11} 
\begin{equation}\label{E4}
\|L_A\|_{\chi}=\chi(AF)=\lim_{s \rightarrow\infty} \left[\sup_{x\in F}\|(I-P_s)Ax\|\right]
\end{equation}
Again using Proposition \ref{P3} and Corollary \ref{C1} we have 
\begin{equation}\label{E5}
\|A\|^{s}=\sup_{x\in F}\|(I-P_s)Ax\|
\end{equation}
From \eqref{E4} and \eqref{E5}  we get 
\begin{align*}
\|L_A\|_{\chi}=\lim_{s \rightarrow\infty} \|A\|^{(s)}.
\end{align*}
Since every sequence $x=(x_k)_{k=0}^{\infty} \in c$ has a unique representation 
$$x=le+\sum_{k=0}^{\infty} (x_k-l)e^{(k)}~~~~~~~~~~~\text{where} ~~l\in \mathbb{C}~~\text{is such that } x-le\in c_0$$
We define $P_s:c\rightarrow c$ by  $P_s(x)=le+\sum_{k=0}^{s} (x_k-l)e^{(k)}$, $s=0,1,2,\ldots$.\\
Then $\|I-P_s\|=2$ and using \eqref{E5}  and Proposition \ref{P12} we get 
\begin{align*}
{1\over 2}\cdot \lim_{s\rightarrow\infty} \|A\|^{(s)}\leq \|L_A\|_{\chi}\leq \lim_{s\rightarrow\infty} \|A\|^{(s)}
\end{align*}
Finally we define $P_s:\ell_\infty\rightarrow \ell_\infty$ by $P_s(x)=(x_0,x_1,\ldots, x_s, 0,0\ldots)$, $x=(x_k)\in \ell_\infty$.\\
Clearly $ AF\subset P_s(AF)+(I-P_s)(AF)$\\
So using the properties of $\chi$ we get 
\begin{align*}
\chi(AF)&\leq  \chi[P_s(AF)]+\chi[(I-P_s)(AF)]\\
&=\chi[(I-P_s)(AF)] \\
&\leq \sup_{x\in F}\|(I-P_s)A(x)\|
\end{align*}
Hence by Proposition \ref{P3} and and Corollary \ref{C1} we get
\begin{align*}
0\leq \|L_A\|_{\chi}\leq \lim_{s\rightarrow\infty} \|A\|^{(s)}
\end{align*}
\end{proof}

A direct corollary  of the above theorem is
\begin{corollary}
Consider the matrix $A$ as in Corollary \ref{C1}, and $X=\left(N_{\Delta^-}^{q} \right)_0$ or $X=\left(N_{\Delta^-}^{q} \right)$  
then if $A\in (X,c_0)$ or $A\in (X,c)$ we have 
\begin{align*}
L_A \text{  is compact if and only if } \lim_{s\rightarrow \infty}\|A\|^{(s)}=0 
\end{align*}
 Further, for $X=\left(N_{\Delta^-}^{q} \right)_0$ , $X=\left(N_{\Delta^-}^{q} \right)$ or $X=\left(N_{\Delta^-}^{q} \right)_\infty$, if $A\in (X,\ell_\infty)$ then we have   
\begin{align}\label{C111}
L_A \text{  is compact if } \lim_{s\rightarrow \infty}\|A\|^{(s)}=0 
\end{align}
\end{corollary}
In \eqref{C111} it is possible for $L_A$ to be compact although $\lim_{s\rightarrow \infty}\|A\|^{(s)}\not=0$, that is the condition is only sufficient condition for $L_A$ to be compact.\\
For example, let the matrix $A$ be defined as $A_n=e^{(1)}~~~n=0,1,2,\ldots$ and $q^n=3^n~,~n=0,1,2,\ldots$ . \\
Then by \eqref{Co1i} we have 
$$\sup_{m,n}\left[\sum_{k=0}^{m-1}Q_{k}\left|\left({1\over q_{k+1}}-{1\over q_k}\right) \sum_{j=k+1}^{m}a_{nj}\right|+
\left|{Q_ma_{nm}\over q_m}\right|\right]=\sup_{n}\left( {2\over3}+{1\over 2}(1-3^{-n})\right)<2$$
Now by Corollary \ref{C1} we know $A\in \left( \left(N_{\Delta^-}^{q} \right)_\infty,\ell_\infty\right)$ . \\
But 
$$\|A\|^{(s)}=\sup_{n>s}\left[ {2\over3}+{1\over 2}(1-3^{-n}) \right]={7\over 6}-{1\over 2\cdot 3^{r+1}}~~~\forall ~r$$
Which gives $\|A\|^{(s)} ={7\over 6}\not=0$.\\
Since $A(x)=x_1$ for all $x\in \left(N_{\Delta^-}^{q} \right)_\infty$, so $L_A$ is  compact operator.

\parindent=5mm\vspace{0.00in} 

\bibliographystyle{unsrt}
\bibliography{References}

\begin{thebibliography}{10}

\bibitem{kiz81}
H~Kizmaz.
\newblock Certain sequence spaces.
\newblock {\em Can. Math. Bull.}, 24(2):169--176, 1981.

\bibitem{mur10}
Mohammad Mursaleen and Abdullah~K Noman.
\newblock On some new difference sequence spaces of non-absolute type.
\newblock {\em Mathematical and Computer Modelling}, 52(3-4):603--617, 2010.

\bibitem{sir92}
SM~Sirajudeen.
\newblock Matrix transformations of bv into.
\newblock {\em Indian J. pure appl. Math}, 23(1):55--61, 1992.

\bibitem{alt02}
Bilal Altay and Feyzi Basar.
\newblock On the paranormed riesz sequence spaces of non-absolute type.
\newblock {\em Southeast Asian Bull. Math}, 26(5):701--715, 2002.

\bibitem{M1-b}
Jozef Banas and Mohammad Mursaleen.
\newblock {\em Sequence spaces and measures of noncompactness with applications
  to differential and integral equations}.
\newblock Springer, 2014.

\bibitem{wilansky2000summability}
Albert Wilansky.
\newblock {\em Summability through functional analysis}, volume~85.
\newblock Elsevier, 2000.

\bibitem{Al-Em03}
Abdullah~M Jarrah and Eberhard Malkowsky.
\newblock Ordinary, absolute and strong summability and matrix transformations.
\newblock {\em Filomat}, pages 59--78, 2003.

\bibitem{maddox88}
Ivor~John Maddox.
\newblock {\em Elements of functional analysis}.
\newblock CUP Archive, 1988.

\bibitem{Al-Em}
AM~Al-Jarrah and E~Malkowsky.
\newblock Bk spaces, bases and linear operators.
\newblock {\em Rend. del Circ. Mat. di Palermo. Serie II. Suppl}, 52:177--191,
  1998.

\bibitem{mal-Nq01}
Eberhard Malkowsky and V~Rakocevic.
\newblock Measure of noncompactness of linear operators between spaces of
  sequences that are (n, q) summable or bounded.
\newblock {\em Czechoslovak Mathematical Journal}, 51(3):505--522, 2001.

\bibitem{st1}
Michael Stieglitz and Hubert Tietz.
\newblock Matrixtransformationen von folgenr{\"a}umen eine
  ergebnis{\"u}bersicht.
\newblock {\em Mathematische Zeitschrift}, 154(1):1--16, 1977.

\bibitem{mal98}
E~Malkowsky and V~Rakocevic.
\newblock The measure of noncompactness of linear operators between certain
  sequence spaces.
\newblock {\em Acta Scientiarum Mathematicarum}, 64(1):151--170, 1998.

\bibitem{mal99}
E~Malkowsky and V~Rakocevic.
\newblock The measure of noncompactness of linear operators between spaces of
  $m$th-order difference sequences.
\newblock {\em Studia Scientiarum Mathematicarum Hungarica}, 35(3-4):381--396,
  1999.

\bibitem{mur11}
Mohammad Mursaleen, Vatan Karakaya, Harun Polat, and N~Sim{\c{s}}ek.
\newblock Measure of noncompactness of matrix operators on some difference
  sequence spaces of weighted means.
\newblock {\em Computers \& Mathematics with Applications}, 62(2):814--820,
  2011.

\bibitem{Josaf}
J.~Banas and K.~Goebl.
\newblock {\em Measures of noncompactness in Banach spaces. Lecture Notes in
  Pure and Appl. Math.}
\newblock Number~60. 1980.

\bibitem{malkowsky1}
Eberhard Malkowsky and Vladimir Rakocevi{\'c}.
\newblock {\em An introduction into the theory of sequence spaces and measures
  of noncompactness}.
\newblock Number~17. Matemati{\v{c}}ki institut SANU, 2000.

\end{thebibliography}

{\small

\noindent{\bf Tanweer Jalal}

\noindent Department of Mathematics

\noindent Associate Professor of Mathematics

\noindent National Institute of Technology, Srinagar 

\noindent Srinagar, India

\noindent E-mail: tjalal@nitsri.net}\\

{\small
\noindent{\bf  Ishfaq Ahmad Malik}

\noindent  Department of Mathematics

\noindent Research Scholar

\noindent National Institute of Technology, Srinagar

\noindent Srinagar, India

\noindent E-mail: ishfaq$\_$2phd15@nitsri.net}\\

\end{document}